\documentclass[reqno]{amsart}

\usepackage{epsfig}
\usepackage{graphics}
\usepackage{amsmath}
\usepackage{amsthm}
\usepackage{amssymb,enumerate}
\usepackage{latexsym}

\newtheorem{theorem}{Theorem}[section]
\newtheorem{lemma}[theorem]{Lemma}
\theoremstyle{definition}
\newtheorem{definition}[theorem]{Definition}
\newtheorem{example}{Example}[section]
\newtheorem{remark}[theorem]{Remark}

\begin{document}

\title{Generalized Fuzzy Euler--Lagrange equations and transversality conditions}

\author{O.S. Fard, R. Almeida, J. Soolaki, A.H. Borzabadi}

\address{Omid Solaymani Fard\newline
\indent School of Mathematics and Computer Science, Damghan University, Damghan, Iran}
\email{osfard@du.ac.ir, omidsfard@gmail.com}

\address{Ricardo Almeida\newline
\indent Center for Research and Development in Mathematics and Applications (CIDMA)\\
Department of Mathematics, University of Aveiro, 3810--193 Aveiro, Portugal}
\email{ricardo.almeida@ua.pt}

\address{Javad Soolaki \newline
\indent Department of Mathematics, Damghan University, Damghan, Iran}
\email{javad.soolaki@gmail.com}

\address{Akbar Hashemi Borzabadi\newline
\indent School of Mathematics and Computer Science, Damghan University, Damghan, Iran}
\email{borzabadi@du.ac.ir}


\subjclass[2010]{93C42 (Primary); 26A33, 34A07, 34A08, 93D05 (Secondary).}

\keywords{Fuzzy fractional variational problems, Fuzzy fractional Euler--Lagrange conditions, Fuzzy Liouville--Caputo derivative.}


\begin{abstract}
The study of fuzzy fractional variational problems in terms of a fractional Liouville--Caputo derivative is introduced. Necessary  optimality conditions for problems of the fuzzy fractional calculus of variations with  free end--points are proved, as well as transversality conditions.
\end{abstract}

\maketitle


\section{Introduction}

 Fractional calculus is one of the generalizations of classical calculus and it has been used successfully in various fields
of science and engineering. The   real phenomena, such as dielectric and electrode-electrolyte polarization, electromagnetic waves, earthquakes, fluid dynamics, traffic, viscoelasticity and viscoplasticity, can be described successfully and more accurately using fractional models (see Kilbas et al. \cite{Kilbas}, Podlubny \cite{Podlubny}).

 Fractional variational calculus is one of the areas where considerable progress has been made. Riewe \cite{Riewe} was the first to initiate this
field. Afterward, Agrawal \cite{Agrawal} combined calculus of variations with fractional
derivatives to develop Euler--Lagrange equations for fractional variational problems. The generalized Euler--Lagrange equation	 and transversality conditions for fractional variational problems, in
terms of the Riemann--Liouville and the Liouville--Caputo derivatives, were presented in \cite{Agrawal1,Agrawal2} respectively. The necessary and sufficient optimality conditions for  the
fractional calculus of variations problems were also derived in \cite{Torres2,Torres,Torres1,Torres3}. The books written by  Malinowska et al. \cite{book:IFCV} and  Almeida et.al \cite{book:CMFCV}, and also
Odzijewicz et al. \cite{book:AMFCV}, give many useful techniques
for solving the fractional variational problems. \par
The fuzzy calculus of variations  extends
the classical variational calculus  considering  variables and their derivatives in fuzzy form. Although the notion of fuzzy set is widely spread to various problems such as   optimization,
differential equations and even fractional differential equations, very few works has been done to the fuzzy variational problems \cite{Fard,Fard1,Farhadinia}. Recently, Farhadinia \cite{Farhadinia} studied necessary optimality conditions for fuzzy variational problems using the fuzzy differentiability concept due to Buckley and Feuring \cite{Buckley}. Farhadinia's work was generalized by Fard et al. \cite{Fard,Fard1}. In \cite{Fard}  Fard and
Zadeh, using a $\alpha$--differentiability concept, obtained an extended fuzzy
Euler--Lagrange condition.  Fard et al.\cite{Fard1} presented the fuzzy Euler--Lagrange condition  for fuzzy constrained and unconstrained variational problems under the generalized Hukuhara differentiability.\par
Allahviranloo et al. in \cite{Allahviranloo2,Allahviranloo1}
have studied the concepts about generalized Hukuhara fractional Riemann--Liouville and
Liouville--Caputo differentiability of fuzzy--valued functions.  Fard et al.\cite{salehi} investigated the fuzzy fractional Euler--Lagrange equation  for  the following functional
\begin{equation}
\tilde{J}(\tilde{y})=\int_{a}^{b} \tilde{L}\left(x,\tilde{y}(x),_a^{C}\hspace{-0.10cm}D{_{x}^{\alpha}}\tilde{y}(x),_x^{C}\hspace{-0.10cm}D{_{b}^{\beta}}\tilde{y}(x)\right) dx.
\end{equation}
They considered the simple case where the lower bound  (or upper bound) of $\tilde{J}$      is stated in terms containing only $\underline{y}^r (x),$$_a^C~\hspace{-0.10cm}D_{x}^{\alpha}\underline{y}^r(x),$ and $_x^C~\hspace{-0.10cm}D_{b}^{\beta}\underline{y}^r(x)$ (or $\overline{y}^r (x),$ $_a^C~\hspace{-0.10cm}D_{x}^{\alpha}\overline{y}^r(x),$ and $_x^C~\hspace{-0.10cm}D_{b}^{\beta}\overline{y}^r(x)).$  In this article the fuzzy Euler--Lagrange
condition  in \cite{salehi} is generalized, such that both lower bound and upper bound   of the functional are considered to be in
terms containing $\underline{y}^r (x),_a^C~\hspace{-0.10cm}D_{x}^{\alpha}\underline{y}^r(x),_x^C~\hspace{-0.10cm}D_{b}^{\beta}\underline{y}^r(x),\overline{y}^r (x)$, $_a^C~\hspace{-0.10cm}D_{x}^{\alpha}\overline{y}^r(x)$ and $_x^C~\hspace{-0.10cm}D_{b}^{\beta}\overline{y}^r(x)$. Moreover, we develop the theory further
by proving the necessary optimality conditions for more general problems of the fuzzy fractional calculus of variations with a Lagrangian that may also depend on the unspecified end--points. We also  find the transversality conditions for when an end point lies on a given arbitrary
curve.  \par
The paper is organized as follows. Section \ref{1} presents some notations on the fuzzy numbers space, differentiability and integrability
of a fuzzy mapping.  The main results concerning the fuzzy Euler--Lagrange equation  for the FFVPs are established in Section \ref{3}.
Section \ref{4} presents the transversality condition for when
an end point lies on a given arbitrary curve, and finally, conclusions are discussed in Section \ref{6}.


\section{Preliminaries}
\label{1}
Let us denote by $\mathbb{R}_{f}$ the class of fuzzy numbers, i.e., normal, convex,
upper semicontinuous and compactly supported fuzzy subsets of the real numbers.
For $0 < r \leq 1$, let $[\tilde{u}]^{r}=\{{x} \in \mathbb{R}| \, \tilde{u}(x)\geq r\}$
and $[\tilde{u}]^{0}=\overline{\left\{{x} \in \mathbb{R}| \,\tilde{u}(x)\geq 0\right\}}$.
Then, it is well known that $[\tilde{u}]^{r}$ is a bounded closed interval
for any $r\in[0,1]$.
\begin{lemma}[See Theorem~1.1 of \cite{MR0825618} and Lemma~2.1 of \cite{MR2609258}]\label{lem1}
If $\underline{a}^r:[0,1] \rightarrow \mathbb{R}$ and
$\overline{a}^r:[0,1] \rightarrow \mathbb{R}$
satisfy the conditions
\begin{enumerate}
\item[(i)]  $\underline{a}^r:[0,1] \rightarrow \mathbb{R} $ is a bounded nondecreasing function,
\item[(ii)]$\overline{a}^r:[0,1] \rightarrow \mathbb{R} $ is a bounded nonincreasing function,
\item[(iii)]  $\underline{a}^1\leq \overline{a}^1$,
\item[(iv)] for $0<k\leq 1$, $\lim_{r\rightarrow k^{-}}\underline{a}^r =\underline{a}^k$
and $\lim_{r\rightarrow k^{-}}\overline{a}^r =\overline{a}^k$,
\item[(v)] $\lim_{r\rightarrow 0^{+}}\underline{a}^r =\underline{a}^0$
and $\lim_{r\rightarrow 0^{+}}\overline{a}^r =\overline{a}^0$,
\end{enumerate}
then $\tilde{a}:\mathbb{R}\rightarrow [0,1]$, characterized by
$\tilde{a}(t)=\sup\{r| \, \underline{a}^r\leq t \leq \overline{a}^r\}$,
is a fuzzy number with $[\tilde{a}]^{r}=[\underline{a} ^r,\overline{a}^r]$.
The converse is also true: if $\tilde{a}(t)=\sup\{r| \, \underline{a}^r\leq t\leq  \overline{a}^r\}$ is
a fuzzy number with parametrization given by $ [\tilde{a}]^{r}=[\underline{a} ^r,\overline{a}^r]$,
then functions $\underline{a}^r$ and $\overline{a}^r$ satisfy conditions (i)--(v).
\end{lemma}

For $\tilde{u}, \tilde{v}\in\mathbb{R}_{f}$ and $\lambda \in\mathbb{R}$,
the sum $\tilde{u}+\tilde{v}$ and the product $\lambda \cdot \tilde{u}$
are defined by $[\tilde{u}+\tilde{v}]^{r}=[\tilde{u}]^{r}+[\tilde{v}]^{r}$
and $[\lambda \cdotp \tilde{u}]^r=\lambda[\tilde{u}]^r$ for all $r\in[0,1]$,
where $[\tilde{u}]^{r}+[\tilde{v}]^{r}$ means the usual addition of two intervals
(subsets) of $\mathbb{R}$ and $\lambda[\tilde{u}]^r$ means the usual product
between a scalar and a subset of $\mathbb{R}$. The product $\tilde{u} \odot \tilde{v}$
of fuzzy numbers $\tilde{u}$ and $\tilde{v}$, is defined by
$[\tilde{u}\odot\tilde{v}]^r=[\min\{ \underline{u}^r \underline{v}^r,
\underline{u}^r \overline{v}^r, \overline{u}^r \underline{v}^r,
\overline{u}^r \overline{v}^r \},
\max \{\underline{u}^r \underline{v}^r, \underline{u}^r \overline{v}^r,
\overline{u}^r \underline{v}^r, \overline{u}^r \overline{v}^r \}]$.
The metric structure is given by the Hausdorff distance
$D:\mathbb{R}_{f} \times \mathbb{R}_{f} \rightarrow \mathbb{R}_{+}\cup\{0\}$,
$D(\tilde{u},\tilde{v})=\sup_{r\in[0,1]}
\max\{|\underline{u}^r-\underline{v}^r|,|\overline{u}^r-\overline{v}^r|\}$.

We say that the fuzzy number $\tilde{u}$ is triangular if
$\underline{u}^1=\overline{u}^1$, $\underline{u}^r
=\underline{u}^1-(1-r)(\underline{u}^1-\underline{u}^0)$
and $\overline{u}^r=\underline{u}^1-(1-r)(\overline{u}^0-\underline{u}^1)$.
The triangular fuzzy number $u$ is generally denoted by
$\tilde{u}=<\underline{u}^0,\underline{u}^1,\overline{u}^0>$.
We define the fuzzy zero $\tilde{0}_{x}$ as
$$
\tilde{0}_{x}=
\begin{cases}
1 & \text{ if } x=0,\\
0 & \text{ if } x\neq 0.
\end{cases}
$$
\begin{definition}[See \cite{Farhadinia}]
We say that $\tilde{f}:[a,b]\rightarrow \mathbb{R}_{f}$ is continuous
at $x\in[a,b]$, if both $\underline{f}^r(x)$ and $\overline{f}^r(x)$
are continuous functions of $x\in[a,b]$ for all $r \in [0,1]$.
\end{definition}

\begin{definition}[See \cite{Bede5}]
The generalized Hukuhara difference of two fuzzy numbers
$\tilde{x},\tilde{y}\in\mathbb{R}_{f}$ ($gH$-difference for short) is defined as follows:
\begin{equation*}
\tilde{x}\ominus_{gH}\tilde{y}=\tilde{z}
\Leftrightarrow
\tilde{x}=\tilde{y}+\tilde{z}
\text{ or }
\tilde{y}=\tilde{x}+(-1)\tilde{z}.
\end{equation*}
\end{definition}

If $\tilde{z}= \tilde{x}\ominus_{gH}\tilde{y}$ exists as a fuzzy number, then
its level cuts $[\underline{z}^r, \overline{z}^r]$ are obtained by
$\underline{z}^r=\min\{ \underline{x}^r - \underline{y}^r,
\overline{x}^r - \overline{y}^r\}$ and $\overline{z}^r
=\max\{ \underline{x}^r - \underline{y}^r,
\overline{x}^r - \overline{y}^r\}$ for all $r\in[0,1]$.

If the fuzzy function $\tilde{f}(x)$ is continuous in the metric $D$,
then its definite integral exists. Furthermore,
\begin{equation*}
\left( \underline{\int_{a}^{b} \tilde{f}(x)dx} \right)^r=\int_{a}^{b} \underline{f}^r (x)dx,
\quad \left( \overline{\int_{a}^{b} \tilde{f}(x)dx} \right)^r=\int_{a}^{b} \overline{f}^r (x)dx.
\end{equation*}

\begin{definition}[See \cite{Farhadinia}]
\label{def1}
Let $\tilde{a},\tilde{b}\in\mathbb{R}_{f}$. We write $\tilde{a}\preceq\tilde{b}$,
if $\underline{a}^r \leq \underline{b}^r $ and $\overline{a}^r \leq \overline{b}^r$
for all $r \in [0,1]$. We also write $\tilde{a}\prec\tilde{b}$,
if $\tilde{a}\preceq\tilde{b}$ and there exists an $r'\in [0,1]$
so that $\underline{a}^{r'}  < \underline{b}^{r'}$ and
$\overline{a}^{r'}  < \overline{b}^{r'}$.  Moreover,
$\tilde{a}\approx\tilde{b}$ if $\tilde{a}\preceq\tilde{b}$
and $\tilde{a}\succeq\tilde{b}$, that is,
$[\tilde{a}]^r=[\tilde{b}]^r$ for all $r \in[0,1]$.
\end{definition}

We say that $\tilde{a},\tilde{b}\in\mathbb{R}_{f}$
are comparable if either $\tilde{a}\preceq\tilde{b}$
or $\tilde{a}\succeq\tilde{b}$; and noncomparable otherwise.

\begin{definition}
Let $T$ is an open subset of $\mathbb{R}$. A point $x_{0}\in T$ is a locally minimum (resp. maximum) of $\tilde{f}(x)$ if there exists some $\epsilon > 0$ such that $\tilde{f}(x_{0})\preceq\tilde{f}(x)$  (resp. $\tilde{f}(x_{0})\succeq \tilde{f}(x)$) when $x\in N_{\epsilon}(x_{0})$.
\end{definition}

\begin{theorem} [See \cite{brunt}] \label{the3}
Let $f$ be a real--valued function differentiable on the open
interval $I$. If $f$ has a local extremum at $x\in I$, then  $\frac{df}{dx}(x)=0$.
\end{theorem}

\begin{lemma}\label{pro01}
If $x^*$ is a locally  minimum of $\tilde{f}(x)$, then
 $x^*$ is also
a locally  minimum of the real--valued functions $\underline{f}^r(x)$ and  $\overline{f}^r(x)$ for all $r \in [0,1].$ So we have  $\frac{d \underline{f}^r}{d x}(x^*)=\frac{d \overline{f}^r}{d x}(x^*)=0$.
\end{lemma}
\begin{proof}
In the  neighborhood $N_{\epsilon}(x^{*})$ we have $\tilde{f}(x^{*})\preceq\tilde{f}(x)$. Using Definition \ref{def1} we get
\begin{equation*}
\underline{f}^r(x^*) \leq \underline{f}^r(x) ,~
\overline{f}^r(x^*) \leq \overline{f}^r(x)
\end{equation*}
for all $r\in[0,1].$ So $ x^*$
is also a local minimum of the real-valued functions  $\underline{f}^r(x)$  and  $\overline{f}^r(x)$  for all $r\in[0,1]$ and by Theorem \ref{the3} we arrive at
$\frac{d \underline{f}^r}{d x}(x^*)=\frac{d \overline{f}^r}{dx}(x^*)=0.$
\end{proof}

\begin{definition}[See \cite{Farhadinia}]
 A fuzzy  neighborhood $\tilde{N}_{\epsilon }(\tilde{\hat{f}})$ is the set of all curves $\tilde{f} $ satisfying for all~$x\in  [x_{0},x_{f}]$
\begin{equation*}
D(\tilde{\hat{f}}(x),\tilde{f}(x))=\sup_{r \in [0,1]}\max \{|  \hat{\underline{f}}^r (x)- \underline{f}^r (x)|\ , | \overline{\hat{f}}^r(x)-\overline f^r(x)|\}\leq\epsilon
\end{equation*}
where~$\epsilon$ is a small real number.
\end{definition}

\begin{definition}
Let $\tilde{f}$ be a fuzzy function defined on the $\mathbb{R}_f$. The function $\tilde{f}$ is said to have a local minimum at $\tilde{\hat{x}}$ if there exists a number $\epsilon>0$ such that for any $\tilde{x}\in\tilde{N}_{\epsilon }(\tilde{\hat{x}}),\tilde{f}(\tilde{\hat{x}})\preceq \tilde{f}(\tilde{x}).$
\end{definition}

\begin{definition}
 We say that the fuzzy curve $\tilde{f}(x)$ intersects the fuzzy curve $\tilde{g}(x)$ at $x_{0}$ if $\underline{f}^1(x_0)=\underline{g}^1(x_0)$ and $\overline{f}^1(x_0)=\overline{g}^1(x_0)$.
\end{definition}



\section{Optimality for Fuzzy Fractional Variational Problems}\label{3}
Following \cite{Allahviranloo1},  we denote the space of all continuous fuzzy valued functions on $[a,b]\in \mathbb{R}$ by $C^F[a,b];$ the class of fuzzy functions with continuous first derivatives on $[a,b]\in \mathbb{R}$ by $C^{F1}[a,b];$   and the space of all
Lebesgue integrable fuzzy valued functions on the bounded interval $[a,b]$ is indicated by $L^F[a,b].$

Let us consider the following problem:
\begin{equation}\label{e1}
\begin{gathered}
 \tilde{J}(\tilde{y})=\int_{a}^{b} \tilde{L}\Big(x,\tilde{y}(x),   ^{gH-C}\hspace{-0.10cm}_a\mathcal{D}{_{ix}^{\alpha}}\tilde{y}(x),^{gH-C}\hspace{-0.10cm}_x\mathcal{D}{_{ib}^{\beta}}\tilde{y}(x),\\
\tilde{y}(a),\tilde{y}(b)\Big) dx  \longrightarrow \rm{extr},   \quad i = 1, 2,  \\
(\tilde{y}(a)=\tilde{y}_a),~~~(\tilde{y}(b)=\tilde{y}_b).
\end{gathered}
\end{equation}
Here, the Lagrange function $\tilde{L}$ is   assumed to be of class $C^{F1}$ on all its arguments.
To fix notation, we consider $\alpha,\beta\in(0,1)$. Here  $^{gH-C}\hspace{-0.10cm}_a\mathcal{D}{_{ix}^{\alpha}}\tilde{y}(x)$ and $~^{gH-C}\hspace{-0.10cm}_x\mathcal{D}{_{ib}^{\beta}}\tilde{y}(x)$ denote the left Liouville--Caputo Fuzzy Fractional Derivative (LCFFD) and the right Liouville--Caputo Fuzzy Fractional
Derivative (RCFFD), respectively:
\begin{equation*}
^{gH-C}\hspace{-0.10cm}_a\mathcal{D}{_{ix}^{\alpha}}\tilde{y}(x)=\frac{1}{\Gamma(1-\alpha)}\int_{a}^{x} (x-t)^{-\alpha} (\mathcal{D}{_{i,gH}}\tilde{y})(t)dt,
\end{equation*}
\begin{equation*}
^{gH-C}\hspace{-0.10cm}_x\mathcal{D}{_{ib}^{\beta}}\tilde{y}(x)=\frac{-1}{\Gamma(1-\beta)}\int_{x}^{b} (t-x)^{-\beta} (\mathcal{D}{_{i,gH}}\tilde{y})(t)dt
\end{equation*}
for $i=1,2$.
\begin{remark}
We use the notation $^{gH-C}_{\qquad a}\mathcal{D}{_{ix}^{\beta}} \tilde{f}$
when the fuzzy-valued function $\tilde{f}$ is $[(i)-gH]_{\alpha}^{C}$-differentiable
with respect to the independent variable $x$, $i \in \{1,2\}$.
\end{remark}

\begin{definition}[See \cite{Allahviranloo2}]
Let $\alpha\in[0,1]$ and $\tilde{f}:[a,b]\rightarrow \mathbb{R}_f$ be
$[gH]_{\alpha}^{C}$-differentiable at $x \in[a,b]$. We say that
$\tilde{f}$ is $[(1)-gH]_{\alpha}^{C}$-differentiable at $x \in[a,b]$ if
\begin{equation*}
[^{gH-C}_{\qquad a}\mathcal{D}_x^{\alpha}\tilde{f}(x)]^r
=\left[^{C}_aD_x^{\alpha}\underline{f}^r(x),
{^{C}_aD_x^{\alpha}}\overline{f}^r(x)\right],
\quad 0\leq r \leq 1,
\end{equation*}
and that $\tilde{f}$ is $[(2)-gH]_{\alpha}^{C}$-differentiable at $x $ if
\begin{equation*}
[^{gH-C}_{\qquad a}\mathcal{D}_x^{\alpha}\tilde{f}(x)]^r
=\left[^{C}_aD_x^{\alpha}\overline{f}^r(x),
{^{C}_aD_x^{\alpha}}\underline{f}^r(x)\right],
\quad 0\leq r \leq 1.
\end{equation*}
\end{definition}

\begin{remark}
We use the notation $^{gH-C}_{\qquad a}\mathcal{D}{_{ix}^{\beta}} \tilde{f}$
when the fuzzy-valued function $\tilde{f}$ is $[(i)-gH]_{\alpha}^{C}$-differentiable
with respect to the independent variable $x$, $i \in \{1,2\}$.
\end{remark}

 Along the work we denote by
$\partial_i L$ the partial derivative of function $L$ with respect to its
$i$th argument.
To develop the necessary conditions for the extremum for (\ref{e1}), assume that $\tilde{y}^*(x)$ is the
desired function, let $\epsilon \in \mathbb{R},$ and define a family of curves  $\tilde{y}(x)=\tilde{y}^*(x)+\epsilon \tilde{h}(x)$, where  $ \tilde{h}$ is an arbitrary admissible variation. We do not require $\tilde{h}(a)=\tilde{0}$ or $\tilde{h}(b)=\tilde{0}$ in the case when $\tilde{y}(a)$ or $\tilde{y}(b)$, respectively, is free (it is possible that both are free). Let
\begin{equation*}
\begin{gathered}
\tilde{J}(\epsilon)=\int_{a}^{b} \tilde{L}\left(x,\tilde{y}^*(x)+\epsilon \tilde{h}(x), ^{gH-C}\hspace{-0.10cm}_a\mathcal{D}{_{ix}^{\alpha}}\left(\tilde{y}^*+\epsilon\tilde{h}(x)\right),
^{gH-C}\hspace{-0.10cm}_x\mathcal{D}{_{ib}^{\beta}}\left(\tilde{y}^*+\epsilon\tilde{h}(x)\right)\right.,\\
\left.\tilde{y}^*(a)+\epsilon \tilde{h}(a),\tilde{y}^*(b)+ \epsilon \tilde{h}(b)\right)dx,
\end{gathered}
\end{equation*}
for $i=1,2$. The lower bound and upper bound of $\tilde{J}$ are
\begin{align*}
\underline{J}^r(\epsilon)=\int_{a}^{b}\left\{\underline{L}^r\left[x,\tilde{y}^{*}(x)+\epsilon \tilde{h}(x),\tilde{y}^*(a)+ \epsilon \tilde{h}(a),\tilde{y}^*(b)+ \epsilon \tilde{h}(b)\right]^r\right\} dx
\end{align*}
and
\begin{align*}
\overline{J}^r(\epsilon)=\int_{a}^{b} \left\{\overline{L}^r\left[x,\tilde{y}^{*}(x)+\epsilon \tilde{h}(x),\tilde{y}^*(a)+ \epsilon \tilde{h}(a),\tilde{y}^*(b)+ \epsilon \tilde{h}(b)\right]^r\right\}  dx.
\end{align*}
respectively, where
\begin{align*}
\left[x,\tilde{y}^{*}(x)   + \epsilon \tilde{h}(x)\right.&\left.,\tilde{y}^*(a)+ \epsilon \tilde{h}(a),\tilde{y}^*(b)+ \epsilon \tilde{h}(b)\right]^r\\=\left(x,\right.&\underline{y}^{*r}(x)+\epsilon \underline{h}^r(x), \overline{y}^{*r}(x)+\epsilon \overline{h}^r(x),_a^{C}\hspace{-0.10cm}D{_{x}^{\alpha}}\left(\underline{y}^{*r}+ \epsilon\underline{h}^r(x)\right),\\
&_a^{C}\hspace{-0.10cm}D{_{x}^{\alpha}}\left(\overline{y}^{*r}+ \epsilon\overline{h}^r(x)\right),_x^{C}\hspace{-0.10cm}D{_{b}^{\beta}}\left(\underline{y}^{*r}+ \epsilon\underline{h}^r(x)\right), _x^{C}\hspace{-0.10cm}D{_{b}^{\beta}}\left(\overline{y}^{*r}+ \epsilon\overline{h}^r(x)\right),\\
&\left.\underline{y}^r(a)+ \epsilon \underline{h}^r(a), \overline{y}^{*r}(a)+ \epsilon \overline{h}^r(a), ~\underline{y}^{*r}(b)+ \epsilon \underline{h}^r(b),\overline{y}^{*r}(b)+ \epsilon \overline{h}^r(b)\right).
\end{align*}

By Lemma \ref{pro01}, $\tilde{J}(\epsilon)$ is extremum at $\epsilon=0,$   therefore   necessary conditions for $\tilde{y}$ to be an extremizer are given by
setting $\frac{d\underline{J}^r}{d\epsilon}=0,~~\frac{d\overline{J}^r}{d\epsilon}=0,$ at $\epsilon=0$,  for all admissible $\tilde{h}(x)$.
Differentiating $\underline{J}^r$, we obtain

$\frac{d\underline{J}^r}{d\epsilon}|_{\epsilon=0}=0 \longrightarrow$
\begin{align}
\int_{a}^{b}  & \left[\partial _2 \underline{L}^r(...) \underline{h}^r + \partial _3 \underline{L}^r(...) \overline{h}^r +  \partial _4 \underline{L}^r(...) _a^{C}\hspace{-0.10cm}D{_{x}^{\alpha}} \underline{h}^r(x) + \partial _5 \underline{L}^r(...) _a^{C}\hspace{-0.10cm}D{_{x}^{\alpha}}\overline{h}^r(x) \right. \nonumber\\ &+ \partial _6 \underline{L}(...) _x^{C}\hspace{-0.10cm}D{_{b}^{\beta}} \underline{h}^r(x) \nonumber
+ \partial _7 \underline{L}^r(...) _x^{C}\hspace{-0.10cm}D{_{b}^{\beta}}\overline{h}^r(x)+\partial _8 \underline{L}^r(...)\underline{h}^r(a) \\&+\left.\partial _9 \underline{L}^r(...)   \overline{h}^r(a)+ \partial _{10} \underline{L}(...)  \underline{h}^r(b)+ \partial _{11}\underline{L}^r(...)  \overline{h}^r(b)\right]dx=0  \label{e2}
\end{align}
where
\begin{align*}
(...)=(x,\underline{y}^{*r}(x),\overline{y}^{*r}(x),& _a^{C}\hspace{-0.10cm}D{_{x}^{\alpha}}\underline{y}^{*r},_a^{C}\hspace{-0.10cm}D{_{x}^{\alpha}}\overline{y}^{*r},
_x^{C}\hspace{-0.10cm}D{_{b}^{\beta}}\underline{y}^{*r},\\
&_x^{C}\hspace{-0.10cm}D{_{b}^{\beta}}\overline{y}^{*r},\underline{y}^{*r}(a),\overline{y}^{*r}(a), ~\underline{y}^{*r}(b), \overline{y}^{*r}(b)).
\end{align*}
We consider   Eq. (\ref{e2}). Using  integration by parts,
\begin{equation}\label{e4}
\begin{gathered}
\int_{a}^{b} \partial _4 \underline{L}^r(...) _a^{C}\hspace{-0.10cm}D{_{x}^{\alpha}} \underline{h}^r(x)dx= \int_{a}^{b} \underline{h}^r(x) _x~\hspace{-0.10cm}D{_{b}^{\alpha}} \partial _4 \underline{L}^r(...)  dx +   _x\hspace{-0.10cm}I{_{b}^{1-\alpha}}  \partial _4 \underline{L}^r(...) \underline{h}^r(x)|_{x=a}^{x=b},
\\
\int_{a}^{b} \partial _5 \underline{L}^r(...) _a^{C}\hspace{-0.10cm}D{_{x}^{\alpha}} \overline{h}^r(x)dx= \int_{a}^{b} \overline{h}^r(x) _x~\hspace{-0.10cm}D{_{b}^{\alpha}} \partial _5 \underline{L}^r(...)  dx +   _x\hspace{-0.10cm}I{_{b}^{1-\alpha}}  \partial _5 \underline{L}^r(...) \overline{h}^r(x)|_{x=a}^{x=b},
\\
\int_{a}^{b} \partial _6 \underline{L}^r(...) _x^{C}\hspace{-0.10cm}D{_{b}^{\beta}} \underline{h}^r(x)dx= \int_{a}^{b} \underline{h}^r(x) _a~\hspace{-0.10cm}D{_{x}^{\beta}} \partial _6 \underline{L}^r(...)  dx -  _a\hspace{-0.10cm}I{_{x}^{1-\beta}}  \partial _6\underline{L}^r(...) \underline{h}^r(x)|_{x=a}^{x=b},
\\
\int_{a}^{b} \partial _7 \underline{L}^r(...) _x^{C}\hspace{-0.10cm}D{_{b}^{\beta}} \overline{h}^r(x)dx= \int_{a}^{b} \overline{h}^r(x) _a~\hspace{-0.10cm}D{_{x}^{\beta}} \partial _7 \underline{L}^r(...)  dx -  _a\hspace{-0.10cm}I{_{x}^{1-\beta}}  \partial _7\underline{L}^r(...) \overline{h}^r(x)|_{x=a}^{x=b}.
\end{gathered}
\end{equation}
Substituting Eqs. (\ref{e4}) into Eq. (\ref{e2}), we get
\begin{align}\label{e8}
&\int_{a}^{b}  \left\{ \left[\partial_2   \underline{L}^r(...)+  _x~\hspace{-0.10cm}D{_{b}^{\alpha}} \partial_4 \underline{L}^r(...) + _a ~\hspace{-0.10cm}D{_{x}^{\beta}} \partial_6 \underline{L}^r(...)\right]\underline{h}^r(x)+[\partial_3 \underline{L}^r(...)+ _x\hspace{-0.10cm}D{_{b}^{\alpha}} \partial_5 \underline{L}^r(...)\nonumber \right.\\
&\quad\left. +_a~\hspace{-0.10cm}D{_{x}^{\beta}} \partial_7 \underline{L}^r(...)]\overline{h}^r(x)\right\}dx
+ \left( _x ~\hspace{-0.10cm}I{_{b}^{1-\alpha}}  \partial _4 \underline{L}^r(...) -  _a\hspace{-0.10cm}I{_{x}^{1-\beta}}  \partial _6\underline{L}^r(...)\right)  \underline{h}^r(x)|_{x=a}^{x=b}\nonumber \\
&\quad+\left(_x~\hspace{-0.10cm}I{_{b}^{1-\alpha}}  \partial _5 \underline{L}^r(...)- _a\hspace{-0.10cm}I{_{x}^{1-\beta}}  \partial _7\underline{L}^r(...)\right) \overline{h}^r(x)|_{x=a}^{x=b} +\int_{a}^{b} \left\{\partial _8 \underline{L}^r(...)\underline{h}^r(a)\right. \nonumber \\
&\quad+\left.\partial _9 \underline{L}^r(...)   \overline{h}^r(a)+\partial _{10} \underline{L}^r(...)  \underline{h}^r(b)+\partial _{11}\underline{L}^r(...)  \overline{h}^r(b)\right\}dx=0.
\end{align}
We first consider functions $\underline{h}^r$ and $\overline{h}^r$ such that   $\underline{h}^r(a)=\overline{h}^r(a)=\underline{h}^r(b)=\overline{h}^r(b)=0$. Then, by the fundamental lemma of the calculus of variations,
we deduce that
\begin{equation}\label{e9}
\partial_2 \underline{L}^r(...)+  _x\hspace{-0.10cm}D{_{b}^{\alpha}} \partial_4 \underline{L}^r(...) + _a\hspace{-0.10cm}D{_{x}^{\beta}} \partial_6 \underline{L}^r(...)=0,
\end{equation}
\begin{equation}\label{e10}
\partial_3 \underline{L}^r(...)+  _x\hspace{-0.10cm}D{_{b}^{\alpha}} \partial_5 \underline{L}^r(...)+
 _a\hspace{-0.10cm}D{_{x}^{\beta}} \partial_7 \underline{L}^r(...)=0.
\end{equation}
for all $x\in[a,b].$
Therefore, in order for $\tilde{y}^*$ to be an extremizer to the problem (\ref{e1}), $\tilde{y}^*$ must be a solution of the fuzzy fractional
Euler--Lagrange equations (\ref{e9}) and (\ref{e10}). But if $\tilde{y}^*$ is a solution of (\ref{e9}) and (\ref{e10}), the first integral in expression (\ref{e8}) vanishes, then the condition
(\ref{e2}) takes the form
\begin{align*}
&\underline{h}^r(a)\left[\int_{a}^{b} \partial _8 \underline{L}^r(...)dx-( _x ~\hspace{-0.10cm}I{_{b}^{1-\alpha}}  \partial _4 \underline{L}^r(...) -  _a\hspace{-0.10cm}I{_{x}^{1-\beta}}  \partial _6\underline{L}^r(...))|_{x=a}\right]
\\ & + \overline{h}^r(a)\left[\int_{a}^{b} \partial _9 \underline{L}^r(...)dx-( _x ~\hspace{-0.10cm}I{_{b}^{1-\alpha}}  \partial _5 \underline{L}^r(...) -  _a\hspace{-0.10cm}I{_{x}^{1-\beta}}  \partial _7\underline{L}^r(...))|_{x=a}\right]
\\ & +\underline{h}^r(b)\left[\int_{a}^{b} \partial _{10}\underline{L}^r(...)dx+( _x ~\hspace{-0.10cm}I{_{b}^{1-\alpha}}  \partial _4 \underline{L}^r(...) -  _a\hspace{-0.10cm}I{_{x}^{1-\beta}}  \partial _6\underline{L}^r(...))|_{x=b}\right]
\\ & +\overline{h}^r(b)\left[\int_{a}^{b} \partial _{11} \underline{L}^r(...)dx+( _x ~\hspace{-0.10cm}I{_{b}^{1-\alpha}}  \partial _5 \underline{L}^r(...) -  _a\hspace{-0.10cm}I{_{x}^{1-\beta}}  \partial _7\underline{L}^r(...))|_{x=b}\right]=0.
\end{align*}
If $\tilde{y}(a)=\tilde{y}_a$ and $\tilde{y}(b)=\tilde{y}_b$ are given in the formulation of problem (\ref{e1}), then the latter equation is trivially satisfied since $\tilde{h}(a)=\tilde{h}(b)=\tilde{0}$.
Because $\underline{h}^r(a),~\overline{h}^r(a)$ are  arbitrary, when $\tilde{y}(a)$ is free, we have
\begin{equation}\label{e12}
\int_{a}^{b} \partial _8 \underline{L}^r(...)dx-\left( _x ~\hspace{-0.10cm}I{_{b}^{1-\alpha}}  \partial _4 \underline{L}^r(...) -  _a\hspace{-0.10cm}I{_{x}^{1-\beta}}  \partial _6\underline{L}^r(...)\right)|_{x=a}=0,
\end{equation}
\begin{equation}\label{e13}
\int_{a}^{b} \partial _9 \underline{L}^r(...)dx-\left( _x ~\hspace{-0.10cm}I{_{b}^{1-\alpha}}  \partial _5 \underline{L}^r(...) -  _a\hspace{-0.10cm}I{_{x}^{1-\beta}}  \partial _7\underline{L}^r(...)\right)|_{x=a}=0.
\end{equation}
When $\tilde{y}(b)$ is free, we get
\begin{equation}\label{e14}
\int_{a}^{b} \partial _{10}\underline{L}^r(...)dx+\left( _x ~\hspace{-0.10cm}I{_{b}^{1-\alpha}}  \partial _4 \underline{L}^r(...) -  _a\hspace{-0.10cm}I{_{x}^{1-\beta}}  \partial _6\underline{L}^r(...)\right)|_{x=b}=0,
\end{equation}
\begin{equation}\label{e15}
\int_{a}^{b} \partial _{11} \underline{L}^r(...)dx+\left( _x ~\hspace{-0.10cm}I{_{b}^{1-\alpha}}  \partial _5 \underline{L}^r(...) -  _a\hspace{-0.10cm}I{_{x}^{1-\beta}}  \partial _7\underline{L}^r(...)\right)|_{x=b}=0,
\end{equation}
because $\underline{h}^r(b)$ and $\overline{h}^r(b)$ are arbitrary. Following the scheme of obtaining (\ref{e9})-(\ref{e15}) and adapting it to the case under consideration $\frac{d \overline{J}^r}{d \epsilon}=0$, one can show that
\begin{equation*}
\partial_2 \overline{L}^r(...)+  _x\hspace{-0.10cm}D{_{b}^{\alpha}} \partial_4 \overline{L}^r(...) + _a\hspace{-0.10cm}D{_{x}^{\beta}} \partial_6 \overline{L}^r(...)=0,
\end{equation*}
\begin{equation*}
\partial_3 \overline{L}^r(...)+  _x\hspace{-0.10cm}D{_{b}^{\alpha}} \partial_5 \overline{L}^r(...)+
 _a\hspace{-0.10cm}D{_{x}^{\beta}} \partial_7 \overline{L}^r(...)=0,
\end{equation*}
for all $x\in[a,b]$. Moreover, if $\tilde{y}(a)$ is not specified, then
\begin{equation*}
\int_{a}^{b} \partial _8 \overline{L}^r(...)dx-\left( _x ~\hspace{-0.10cm}I{_{b}^{1-\alpha}}  \partial _4 \overline{L}^r(...) -  _a\hspace{-0.10cm}I{_{x}^{1-\beta}}  \partial _6\overline{L}^r(...)\right)|_{x=a}=0,
\end{equation*}
\begin{equation*}
\int_{a}^{b} \partial _9 \overline{L}^r(...)dx-\left( _x ~\hspace{-0.10cm}I{_{b}^{1-\alpha}}  \partial _5 \overline{L}^r(...) -  _a\hspace{-0.10cm}I{_{x}^{1-\beta}}  \partial _7\overline{L}^r(...)\right)|_{x=a}=0,
\end{equation*}
and when $\tilde{y}(b)$ is free, then
\begin{equation*}
\int_{a}^{b} \partial _{10}\overline{L}^r(...)dx+\left( _x ~\hspace{-0.10cm}I{_{b}^{1-\alpha}}  \partial _4 \overline{L}^r(...) -  _a\hspace{-0.10cm}I{_{x}^{1-\beta}}  \partial _6\overline{L}^r(...)\right)|_{x=b}=0,
\end{equation*}
\begin{equation*}
\int_{a}^{b} \partial _{11}\overline{L}^r(...)dx+\left( _x ~\hspace{-0.10cm}I{_{b}^{1-\alpha}}  \partial _5 \overline{L}^r(...) -  _a\hspace{-0.10cm}I{_{x}^{1-\beta}}  \partial _7\overline{L}^r(...)\right)|_{x=b}=0.
\end{equation*}
Now we are in a position to state the necessary conditions for a relative (local) minimum (maximum) of problem (\ref{e1}), as follows:
\begin{theorem}\label{e11}
Let $\tilde{y}^*$ be a local extremizer to problem (\ref{e1}). Then, $\tilde{y}$ satisfies the fractional Euler--Lagrange equations
\[
\begin{cases}
\partial_2 \underline{L}^r(...)+  _x \hspace{-0.10cm}D{_{b}^{\alpha}} \partial_4 \underline{L}^r(...) + _a \hspace{-0.10cm}D{_{x}^{\beta}} \partial_6 \underline{L}^r(...)=0,\\
\partial_2 \overline{L}^r(...)+  _x \hspace{-0.10cm}D{_{b}^{\alpha}} \partial_4 \overline{L}^r(...) + _a \hspace{-0.10cm}D{_{x}^{\beta}} \partial_6 \overline{L}^r(...)=0,\\
\partial_3 \underline{L}^r(...)+  _x \hspace{-0.10cm}D{_{b}^{\alpha}} \partial_5 \underline{L}^r(...)+
 _a \hspace{-0.10cm}D{_{x}^{\beta}} \partial_7 \underline{L}^r(...)=0,
\\
\partial_3 \overline{L}^r(...)+  _x \hspace{-0.10cm}D{_{b}^{\alpha}} \partial_5 \overline{L}^r(...)+
 _a \hspace{-0.10cm}D{_{x}^{\beta}} \partial_7 \underline{L}^r(...)=0.
\end{cases}
\]
for all $x\in[a,b]$. Moreover, if $\tilde{y}(a)$ is free, then
\[
\begin{cases}
\int_{a}^{b} \partial _8 \underline{L}^r(...)dx-( _x ~\hspace{-0.10cm}I{_{b}^{1-\alpha}}  \partial _4 \underline{L}^r(...) -  _a\hspace{-0.10cm}I{_{x}^{1-\beta}}  \partial _6\underline{L}^r(...))|_{x=a}=0,
\\
\int_{a}^{b} \partial _8 \overline{L}^r(...)dx-( _x ~\hspace{-0.10cm}I{_{b}^{1-\alpha}}  \partial _4 \overline{L}^r(...) -  _a\hspace{-0.10cm}I{_{x}^{1-\beta}}  \partial _6\overline{L}^r(...))|_{x=a}=0,
\\
\int_{a}^{b} \partial _9 \underline{L}^r(...)dx-( _x ~\hspace{-0.10cm}I{_{b}^{1-\alpha}}  \partial _5 \underline{L}^r(...) -  _a\hspace{-0.10cm}I{_{x}^{1-\beta}}  \partial _7\underline{L}^r(...))|_{x=a}=0,
\\
\int_{a}^{b} \partial _9 \overline{L}^r(...)dx-( _x ~\hspace{-0.10cm}I{_{b}^{1-\alpha}}  \partial _5 \overline{L}^r(...) -  _a\hspace{-0.10cm}I{_{x}^{1-\beta}}  \partial _7\overline{L}^r(...))|_{x=a}=0,
\end{cases}
\]
and when $\tilde{y}(b)$ is free, then
\[
\begin{cases}
\int_{a}^{b} \partial _{10}\underline{L}^r(...)dx+( _x ~\hspace{-0.10cm}I{_{b}^{1-\alpha}}  \partial _4 \underline{L}^r(...) -  _a\hspace{-0.10cm}I{_{x}^{1-\beta}}  \partial _6\underline{L}^r(...))|_{x=b}=0,
\\
\int_{a}^{b} \partial _{10}\overline{L}^r(...)dx+( _x ~\hspace{-0.10cm}I{_{b}^{1-\alpha}}  \partial _4 \overline{L}^r(...) -  _a\hspace{-0.10cm}I{_{x}^{1-\beta}}  \partial _6 \overline{L}^r(...))|_{x=b}=0,
\\
\int_{a}^{b} \partial _{11}\underline{L}^r(...)dx+( _x ~\hspace{-0.10cm}I{_{b}^{1-\alpha}}  \partial _5 \underline{L}^r(...) -  _a\hspace{-0.10cm}I{_{x}^{1-\beta}}  \partial _7\underline{L}^r(...))|_{x=b}=0,
\\
\int_{a}^{b} \partial _{11}\overline{L}^r(...)dx+( _x ~\hspace{-0.10cm}I{_{b}^{1-\alpha}}  \partial _5 \overline{L}^r(...) -  _a\hspace{-0.10cm}I{_{x}^{1-\beta}}  \partial _7\overline{L}^r(...))|_{x=b}=0,
\end{cases}
\]
where
\begin{align*}
(...)=(x,\underline{y}^{*r}(x),\overline{y}^{*r}(x),&  _a^{C}\hspace{-0.10cm}D{_{x}^{\alpha}}\underline{y}^{*r},_a^{C}\hspace{-0.10cm}D{_{x}^{\alpha}}\overline{y}^{*r},
~_x^{C}\hspace{-0.10cm}D{_{b}^{\beta}}\underline{y}^{*r},\\
&_x^{C}\hspace{-0.10cm}D{_{b}^{\beta}}\overline{y}^{*r},\underline{y}^{*r}(a),\overline{y}^{*r}(a),\underline{y}^{*r}(b),\overline{y}^{*r}(b)).
\end{align*}

\end{theorem}


\begin{example}\label{exa2}
Let us consider the following problem:
\begin{equation}\label{e16}
\tilde{J}(\tilde{y}) = \frac{1}{2}\int_{0}^{1}  \left(^{gH-C}\hspace{-0.10cm}_{0}\mathcal{D}{_{x}^{\alpha}}\tilde{y}(x)\right)^{2} +\tilde{2}\odot\tilde{y}^2(0)+\tilde{3} \odot(\tilde{y}(1)- 1)^2 dx \longrightarrow \min,
\end{equation}
where $\tilde{2}= <1,2,3>,~\tilde{3}= <3,3,3>$.
\end{example}
$Solution$. We first derive the $r$-level set of $\tilde{J}$ for [(1)-gH]-differentiability of $\tilde{y}$
 as follows:
\begin{align*}
\left[\tilde{J}(\tilde{y})\right]^r= \left[ \frac{1}{2}\int_{0}^{1}  \right.&(_{0}^C~ \hspace{-0.10cm}D{_{x}^{\alpha}}\underline{y}^r(x))^{2} +(3-r) (\underline{y}^{r})^2(0)+3(\underline{y}^r(1)-1)^2 dx,\\
&\left. \frac{1}{2}\int_{0}^{1}  (_{0}^C~ \hspace{-0.10cm}D{_{x}^{\alpha}}\overline{y}^r(x))^{2} +(r+1)(\overline{y}^{r})^2(0)+3 (\overline{y}^r(1)-1)^2 dx \right].
\end{align*}
For this problem, the generalized Euler--Lagrange equations and the natural boundary conditions (see Theorem \ref{e11}) are given as
\begin{equation}\label{e17}
\begin{gathered}
_{x}~ \hspace{-0.10cm}D{_{1}^{\alpha}}(_{0}^C~ \hspace{-0.10cm}D{_{x}^{\alpha}}\underline{y}^r(x))=0,
\\
_{x}~ \hspace{-0.10cm}D{_{1}^{\alpha}}(_{0}^C~ \hspace{-0.10cm}D{_{x}^{\alpha}}\overline{y}^r(x))=0,\\
\int_{0}^{1} (3-r) \underline{y}^r(0) dx =_{x} \hspace{-0.10cm}I{_{1}^{1-\alpha}}(_{0}^C~ \hspace{-0.10cm}D{_{x}^{\alpha}}\underline{y}^r(x))|_{x=0},\\
\int_{0}^{1} 3(\underline{y}^r(1)-1) dx = -  _{x}~ \hspace{-0.10cm}I{_{1}^{1-\alpha}}(_{0}^C~ \hspace{-0.10cm}D{_{x}^{\alpha}}\underline{y}^r(x))|_{x=1},\\
\int_{0}^{1} (r+1) \overline{y}^r(0) dx = _{x} \hspace{-0.10cm}I{_{1}^{1-\alpha}}(_{0}^C~ \hspace{-0.10cm}D{_{x}^{\alpha}}\overline{y}^r(x))|_{x=0},\\
\int_{0}^{1} 3(\overline{y}^r(1)-1) dx = -  _{x}~ \hspace{-0.10cm}I{_{1}^{1-\alpha}}(_{0}^C~ \hspace{-0.10cm}D{_{x}^{\alpha}}\overline{y}^r(x))|_{x=1}.
\end{gathered}
\end{equation}
Note that it is difficult to solve the above fractional equations. For $0<\alpha<1$ , a numerical method should be used. When $\alpha$ goes to 1, problem (\ref{e16}) tends to
\begin{equation}\label{e23}
\tilde{J}(\tilde{y}) = \frac{1}{2} \int_{0}^{1}  (\tilde{y}^{'}(x))^{2} +\tilde{2}\odot \tilde{y}^2(0)+\tilde{3}\odot (\tilde{y}(1)- 1)^2 dx \longrightarrow \min,
\end{equation}
and Eqs. (\ref{e17}) could be replaced with
\begin{equation}\label{e24}
\begin{gathered}
(\underline{y}^{r})^{''}(x)=0,
\\
(\overline{y}^{r})^{''}(x)=0,
\\
(3-r)\underline{y}^r(0)=(\underline{y}^{r})^{'}(0),
\\
3(\underline{y}^r(1)-1)=-(\underline{y}^{r}){'}(1),
\\
(r+1)\overline{y}^r(0)=(\overline{y}^{r})^{'}(0),
\\
3(\overline{y}^r(1)-1)=-(\overline{y}^{r})^{'}(1).
\end{gathered}
\end{equation}
By solving Eqs. (\ref{e24}) we have $$\underline{y}^r(x)= \frac{-3r+9} {-4r+15} x +\frac{3} {-4r+15},~~\overline{y}^r(x)= \frac{3r+3} {4r+7} x +\frac{3} {4r+7}.$$  One can easily show that $\underline{y}^r(x)$ and $\overline{y}^r(x)$  satisfy Lemma \ref{lem1}. This solution is shown in Figure \ref{fig2}.

\begin{figure}[tbp]
\centering \includegraphics[scale=.50]{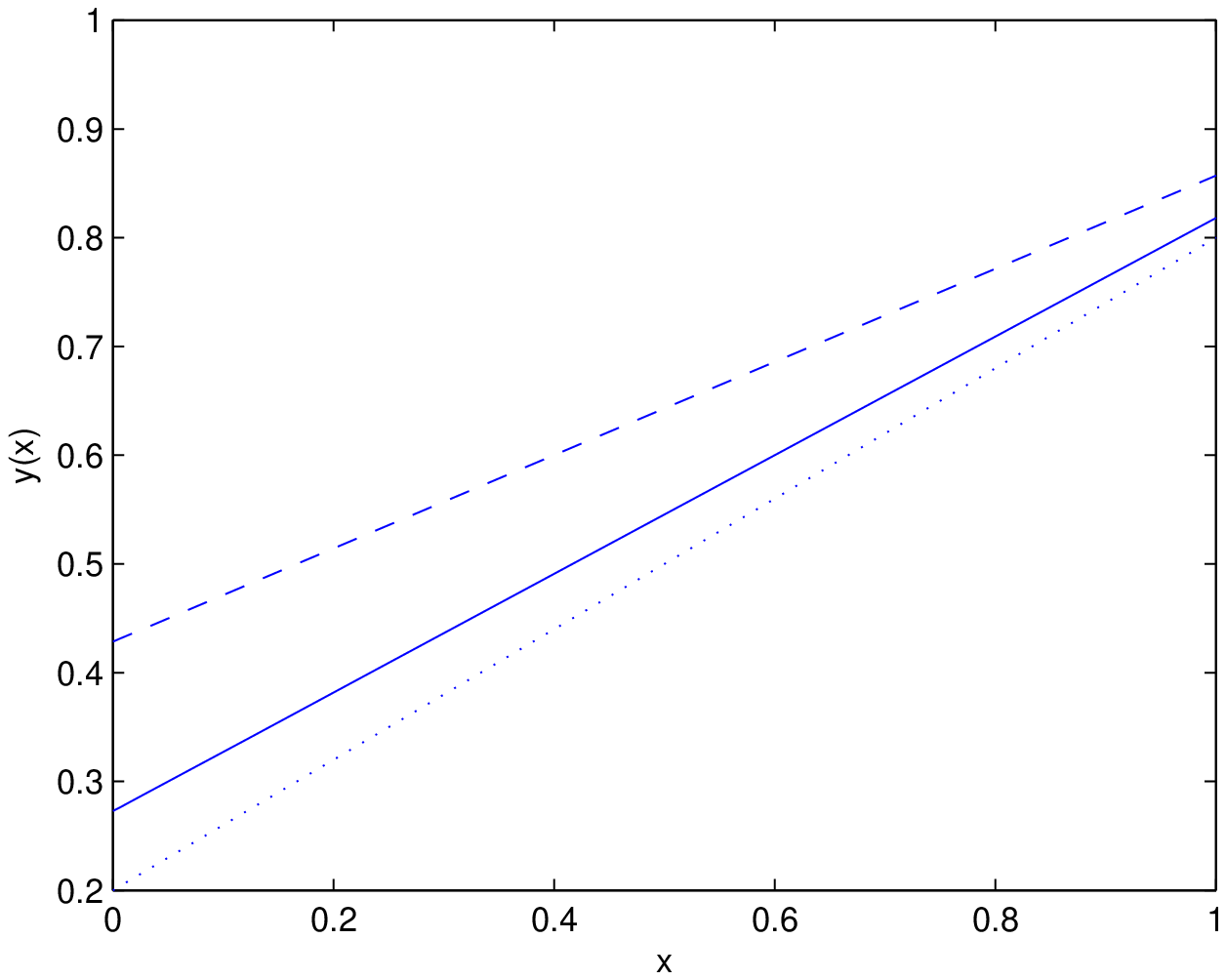}
\caption{Solution of Example \ref{exa2}.}
 $\underline{y}^0(x)$ (dotted line), $\overline{y}^0(x)$   (dashed line),~$\underline{y}^1(x)=\overline{y}^1(x)$ (solid line).\label{fig2}
\end{figure}


We now consider a more general case. We have assumed that the interval of integration of the functional \eqref{e1} is the same of the admissible functions. We generalize Theorem \ref{e11}, by considering a subinterval $[A,B]\subset [a,b]$ and the following optimization problem:
\begin{equation}\label{Ne1}
\begin{gathered}
 \tilde{J}(\tilde{y})=\int_{A}^{B} \tilde{L}\Big(x,\tilde{y}(x),   ^{gH-C}\hspace{-0.10cm}_a\mathcal{D}{_{ix}^{\alpha}}\tilde{y}(x),^{gH-C}\hspace{-0.10cm}_x\mathcal{D}{_{ib}^{\beta}}\tilde{y}(x),\\
\tilde{y}(a),\tilde{y}(A),\tilde{y}(B),\tilde{y}(b)\Big) dx  \longrightarrow \rm{extr},   \quad i = 1, 2,  \\
(\tilde{y}(a)=\tilde{y}_a),~~~(\tilde{y}(A)=\tilde{y}_A),~~~(\tilde{y}(B)=\tilde{y}_B),~~~(\tilde{y}(b)=\tilde{y}_b).
\end{gathered}
\end{equation}
Again, the function $\tilde{L}$ is   assumed to be of class $C^{F1}$ on all its arguments. In a similar way as done before, let $\tilde{y}^*(x)$ be an extremum for \eqref{Ne1}, $\epsilon \in \mathbb{R}$ a real, and consider the variation functions  $\tilde{y}(x)=\tilde{y}^*(x)+\epsilon \tilde{h}(x)$, where  $ \tilde{h}$ is an arbitrary admissible variation, which may satisfy or not the boundary conditions $\tilde{h}(a)=\tilde{0}$, $\tilde{h}(A)=\tilde{0}$, $\tilde{h}(B)=\tilde{0}$ or $\tilde{h}(b)=\tilde{0}$, depending if $\tilde{y}(a)$, $\tilde{y}(A)$, $\tilde{y}(B)$ or $\tilde{y}(b)$, is fixed or not. Define
\begin{equation*}
\begin{gathered}
\tilde{J}(\epsilon)=\int_{A}^{B} \tilde{L}\left(x,\tilde{y}^*(x)+\epsilon \tilde{h}(x), ^{gH-C}\hspace{-0.10cm}_a\mathcal{D}{_{ix}^{\alpha}}\left(\tilde{y}^*+\epsilon\tilde{h}(x)\right),
^{gH-C}\hspace{-0.10cm}_x\mathcal{D}{_{ib}^{\alpha}}\left(\tilde{y}^*+\epsilon\tilde{h}(x)\right)\right.,\\
\left.\tilde{y}^*(a)+\epsilon \tilde{h}(a),\tilde{y}^*(A)+\epsilon \tilde{h}(A),\tilde{y}^*(B)+\epsilon \tilde{h}(B),\tilde{y}^*(b)+ \epsilon \tilde{h}(b)\right)dx,
\end{gathered}
\end{equation*}
for $i=1,2$. The lower bound and upper bound of $\tilde{J}$ are
\begin{align*}
\underline{J}^r(\epsilon)=\int_{A}^{B}\left\{\underline{L}^r\left[x,\tilde{y}^{*}(x)+\epsilon \tilde{h}(x)\right]^r_{a,A,B,b}\right\} dx
\end{align*}
and
\begin{align*}
\overline{J}^r(\epsilon)=\int_{A}^{B} \left\{\overline{L}^r\left[x,\tilde{y}^{*}(x)+\epsilon \tilde{h}(x)\right]^r_{a,A,B,b}\right\}  dx,
\end{align*}
respectively, where
\begin{align*}
\left[x,\tilde{y}^{*}(x)   + \epsilon \tilde{h}(x)\right.&\left.\right]^r_{a,A,B,b}\\
=&\left(x,\right.\underline{y}^{*r}(x)+\epsilon \underline{h}^r(x), \overline{y}^{*r}(x)+\epsilon \overline{h}^r(x),_a^{C}\hspace{-0.10cm}D{_{x}^{\alpha}}\left(\underline{y}^{*r}+ \epsilon\underline{h}^r(x)\right),\\
&_a^{C}\hspace{-0.10cm}D{_{x}^{\alpha}}\left(\overline{y}^{*r}+ \epsilon\overline{h}^r(x)\right),_x^{C}\hspace{-0.10cm}D{_{b}^{\beta}}\left(\underline{y}^{*r}+ \epsilon\underline{h}^r(x)\right), _x^{C}\hspace{-0.10cm}D{_{b}^{\beta}}\left(\overline{y}^{*r}+ \epsilon\overline{h}^r(x)\right),\\
&\underline{y}^r(a)+ \epsilon \underline{h}^r(a), \overline{y}^{*r}(a)+ \epsilon \overline{h}^r(a), ~\underline{y}^{*r}(A)+ \epsilon \underline{h}^r(A),\overline{y}^{*r}(A)+ \epsilon \overline{h}^r(A),\\
&\left.\underline{y}^r(B)+ \epsilon \underline{h}^r(B), \overline{y}^{*r}(B)+ \epsilon \overline{h}^r(B), ~\underline{y}^{*r}(b)+ \epsilon \underline{h}^r(b),\overline{y}^{*r}(b)+ \epsilon \overline{h}^r(b)\right).
\end{align*}

Attending that $\tilde{J}(\epsilon)$ is extremum at $\epsilon=0,$  we have the necessary conditions $\frac{d\underline{J}^r}{d\epsilon}=0,~~\frac{d\overline{J}^r}{d\epsilon}=0,$ at $\epsilon=0$,  for all admissible function $\tilde{h}(x)$. Starting with equation
$$\frac{d\underline{J}^r}{d\epsilon}|_{\epsilon=0}=0,$$
we get
\begin{align}\label{N2}
\int_{A}^{B}  & \left[\partial _2 \underline{L}^r(...) \underline{h}^r + \partial _3 \underline{L}^r(...) \overline{h}^r +  \partial _4 \underline{L}^r(...) _a^{C}\hspace{-0.10cm}D{_{x}^{\alpha}} \underline{h}^r(x) + \partial _5 \underline{L}^r(...) _a^{C}\hspace{-0.10cm}D{_{x}^{\alpha}}\overline{h}^r(x) \right. \nonumber\\
&+ \partial _6 \underline{L}(...) _x^{C}\hspace{-0.10cm}D{_{b}^{\beta}} \underline{h}^r(x) \nonumber
+ \partial _7 \underline{L}^r(...) _x^{C}\hspace{-0.10cm}D{_{b}^{\beta}}\overline{h}^r(x)+\partial _8 \underline{L}^r(...)\underline{h}^r(a)\nonumber \\
&+\partial _9 \underline{L}^r(...)   \overline{h}^r(a)+ \partial _{10} \underline{L}(...)  \underline{h}^r(A)+ \partial _{11}\underline{L}^r(...)  \overline{h}^r(A)\nonumber\\
&+\left.\partial _{12} \underline{L}^r(...)   \underline{h}^r(B)+ \partial _{13} \underline{L}(...)  \overline{h}^r(B)+ \partial _{14}\underline{L}^r(...)  \underline{h}^r(b)+ \partial _{15}\underline{L}^r(...)  \overline{h}^r(b)\right]dx=0,
\end{align}
where $(...)=\left[x,\tilde{y}^{*}(x)\right]^r_{a,A,B,b}.$
Replacing the next four relations
\begin{align*}
\displaystyle \int_{A}^{B}\partial _4 \underline{L}^r(...) _a^{C}\hspace{-0.10cm}D{_{x}^{\alpha}} \underline{h}^r(x)dx
=&\displaystyle \int_{a}^{B}\partial _4 \underline{L}^r(...) _a^{C}\hspace{-0.10cm}D{_{x}^{\alpha}} \underline{h}^r(x)dx
-\int_{a}^{A}\partial _4 \underline{L}^r(...) _a^{C}\hspace{-0.10cm}D{_{x}^{\alpha}} \underline{h}^r(x)dx\\
\displaystyle=& \int_{a}^{B} \underline{h}^r(x) _x~\hspace{-0.10cm}D{_{B}^{\alpha}} \partial _4 \underline{L}^r(...)  dx
-\int_{a}^{A} \underline{h}^r(x) _x~\hspace{-0.10cm}D{_{A}^{\alpha}} \partial _4 \underline{L}^r(...)  dx\\
&\displaystyle+ _x\hspace{-0.10cm}I{_{B}^{1-\alpha}}  \partial _4 \underline{L}^r(...) \underline{h}^r(x)|_{x=a}^{x=B}
- _x\hspace{-0.10cm}I{_{A}^{1-\alpha}}  \partial _4 \underline{L}^r(...) \underline{h}^r(x)|_{x=a}^{x=A},\\
\displaystyle \int_{A}^{B}\partial _5 \underline{L}^r(...) _a^{C}\hspace{-0.10cm}D{_{x}^{\alpha}} \overline{h}^r(x)dx
\displaystyle=& \int_{a}^{B} \overline{h}^r(x) _x~\hspace{-0.10cm}D{_{B}^{\alpha}} \partial _5 \underline{L}^r(...)  dx
-\int_{a}^{A} \overline{h}^r(x) _x~\hspace{-0.10cm}D{_{A}^{\alpha}} \partial _5 \underline{L}^r(...)  dx\\
&\displaystyle+ _x\hspace{-0.10cm}I{_{B}^{1-\alpha}}  \partial _5 \underline{L}^r(...) \overline{h}^r(x)|_{x=a}^{x=B}
- _x\hspace{-0.10cm}I{_{A}^{1-\alpha}}  \partial _5 \underline{L}^r(...) \overline{h}^r(x)|_{x=a}^{x=A},\\
\displaystyle\int_{A}^{B} \partial _6 \underline{L}^r(...) _x^{C}\hspace{-0.10cm}D{_{b}^{\beta}} \underline{h}^r(x)dx
\displaystyle=&\int_{A}^{b} \partial _6 \underline{L}^r(...) _x^{C}\hspace{-0.10cm}D{_{b}^{\beta}} \underline{h}^r(x)dx
-\int_{B}^{b} \partial _6 \underline{L}^r(...) _x^{C}\hspace{-0.10cm}D{_{b}^{\beta}} \underline{h}^r(x)dx\\
\displaystyle = &\int_{A}^{b} \underline{h}^r(x) _A~\hspace{-0.10cm}D{_{x}^{\beta}} \partial _6 \underline{L}^r(...)  dx
\displaystyle -\int_{B}^{b} \underline{h}^r(x) _B~\hspace{-0.10cm}D{_{x}^{\beta}} \partial _6 \underline{L}^r(...)  dx\\
&\displaystyle-  _A\hspace{-0.10cm}I{_{x}^{1-\beta}}  \partial _6\underline{L}^r(...) \underline{h}^r(x)|_{x=A}^{x=b}
+_B\hspace{-0.10cm}I{_{x}^{1-\beta}}  \partial _6\underline{L}^r(...) \underline{h}^r(x)|_{x=B}^{x=b},\\
\displaystyle\int_{A}^{B} \partial _7 \underline{L}^r(...) _x^{C}\hspace{-0.10cm}D{_{b}^{\beta}} \overline{h}^r(x)dx
\displaystyle =& \int_{A}^{b} \overline{h}^r(x) _A~\hspace{-0.10cm}D{_{x}^{\beta}} \partial _7 \underline{L}^r(...)  dx
\displaystyle -\int_{B}^{b} \overline{h}^r(x) _B~\hspace{-0.10cm}D{_{x}^{\beta}} \partial _7 \underline{L}^r(...)  dx\\
&\displaystyle-  _A\hspace{-0.10cm}I{_{x}^{1-\beta}}  \partial _7\underline{L}^r(...) \overline{h}^r(x)|_{x=A}^{x=b}
+_B\hspace{-0.10cm}I{_{x}^{1-\beta}}  \partial _7\underline{L}^r(...) \overline{h}^r(x)|_{x=B}^{x=b},\\
\end{align*}
into Eq. \eqref{N2}, doing the same computations as presented before, and by the arbitrariness of $\tilde{h}$, we obtain the following result.

\begin{theorem}
Let $\tilde{y}^*$ be a local extremizer to problem (\ref{Ne1}). Then, $\tilde{y}$ satisfies the fractional Euler--Lagrange equations
\[
\begin{cases}
_x~\hspace{-0.10cm}D{_{B}^{\alpha}} \partial_4 \underline{L}^r(...)-_x~\hspace{-0.10cm}D{_{A}^{\alpha}} \partial_4 \underline{L}^r(...)=0,\\
_x~\hspace{-0.10cm}D{_{B}^{\alpha}} \partial_4 \overline{L}^r(...)-_x~\hspace{-0.10cm}D{_{A}^{\alpha}} \partial_4 \overline{L}^r(...)=0,\\
_x~\hspace{-0.10cm}D{_{B}^{\alpha}} \partial_5 \underline{L}^r(...)-_x~\hspace{-0.10cm}D{_{A}^{\alpha}} \partial_5 \underline{L}^r(...)=0,\\
_x~\hspace{-0.10cm}D{_{B}^{\alpha}} \partial_5 \overline{L}^r(...)-_x~\hspace{-0.10cm}D{_{A}^{\alpha}} \partial_5 \overline{L}^r(...)=0,\\
\end{cases}
\]
for all $x\in[a,A]$,
\[
\begin{cases}
\partial_2\underline{L}^r(...)+ _x~\hspace{-0.10cm}D{_{B}^{\alpha}} \partial_4 \underline{L}^r(...)+ _A ~\hspace{-0.10cm}D{_{x}^{\beta}} \partial_6 \underline{L}^r(...)=0,\\
\partial_2\overline{L}^r(...)+ _x~\hspace{-0.10cm}D{_{B}^{\alpha}} \partial_4 \overline{L}^r(...)+ _A ~\hspace{-0.10cm}D{_{x}^{\beta}} \partial_6 \overline{L}^r(...)=0,\\
\partial_3 \underline{L}^r(...)+ _x\hspace{-0.10cm}D{_{B}^{\alpha}} \partial_5 \underline{L}^r(...)+_A~\hspace{-0.10cm}D{_{x}^{\beta}} \partial_7 \underline{L}^r(...)=0,\\
\partial_3 \overline{L}^r(...)+ _x\hspace{-0.10cm}D{_{B}^{\alpha}} \partial_5 \overline{L}^r(...)+_A~\hspace{-0.10cm}D{_{x}^{\beta}} \partial_7 \overline{L}^r(...)=0,\\
\end{cases}
\]
for all $x\in[A,B]$,
\[
\begin{cases}
_A~\hspace{-0.10cm}D{_{x}^{\beta}} \partial_6 \underline{L}^r(...)-_B~\hspace{-0.10cm}D{_{x}^{\beta}} \partial_6 \underline{L}^r(...)=0,\\
_A~\hspace{-0.10cm}D{_{x}^{\beta}} \partial_6 \overline{L}^r(...)-_B~\hspace{-0.10cm}D{_{x}^{\beta}} \partial_6 \overline{L}^r(...)=0,\\
_A~\hspace{-0.10cm}D{_{x}^{\beta}} \partial_7 \underline{L}^r(...)-_B~\hspace{-0.10cm}D{_{x}^{\beta}} \partial_7 \underline{L}^r(...)=0,\\
_A~\hspace{-0.10cm}D{_{x}^{\beta}} \partial_7 \overline{L}^r(...)-_B~\hspace{-0.10cm}D{_{x}^{\beta}} \partial_7 \overline{L}^r(...)=0,\\
\end{cases}
\]
for all $x\in[B,b]$. Moreover, if $\tilde{y}(a)$ is free, then
\[
\begin{cases}
\int_{A}^{B} \partial _8 \underline{L}^r(...)dx-(_x ~\hspace{-0.10cm}I{_{B}^{1-\alpha}}  \partial _4 \underline{L}^r(...)- _x\hspace{-0.10cm}I{_{A}^{1-\alpha}}  \partial _4\underline{L}^r(...))|_{x=a}=0,
\\
\int_{A}^{B} \partial _8 \overline{L}^r(...)dx-(_x ~\hspace{-0.10cm}I{_{B}^{1-\alpha}}  \partial _4 \overline{L}^r(...)- _x\hspace{-0.10cm}I{_{A}^{1-\alpha}}  \partial _4\overline{L}^r(...))|_{x=a}=0,
\\
\int_{A}^{B} \partial _9 \underline{L}^r(...)dx-(_x ~\hspace{-0.10cm}I{_{B}^{1-\alpha}}  \partial _5 \underline{L}^r(...)- _x\hspace{-0.10cm}I{_{A}^{1-\alpha}}  \partial _5\underline{L}^r(...))|_{x=a}=0,
\\
\int_{A}^{B} \partial _9 \overline{L}^r(...)dx-(_x ~\hspace{-0.10cm}I{_{B}^{1-\alpha}}  \partial _5 \overline{L}^r(...)- _x\hspace{-0.10cm}I{_{A}^{1-\alpha}}  \partial _5\overline{L}^r(...))|_{x=a}=0;
\end{cases}
\]
when $\tilde{y}(A)$ is free, then
\[
\begin{cases}
\int_{A}^{B} \partial _{10}\underline{L}^r(...)dx-(_x ~\hspace{-0.10cm}I{_{A}^{1-\alpha}}  \partial _4 \underline{L}^r(...)- _A\hspace{-0.10cm}I{_{x}^{1-\beta}}  \partial _6\underline{L}^r(...))|_{x=A}=0,
\\
\int_{A}^{B} \partial _{10}\overline{L}^r(...)dx-(_x ~\hspace{-0.10cm}I{_{A}^{1-\alpha}}  \partial _4 \overline{L}^r(...)- _A\hspace{-0.10cm}I{_{x}^{1-\beta}}  \partial _6\overline{L}^r(...))|_{x=A}=0,
\\
\int_{A}^{B} \partial _{11}\underline{L}^r(...)dx-(_x ~\hspace{-0.10cm}I{_{A}^{1-\alpha}}  \partial _5 \underline{L}^r(...)- _A\hspace{-0.10cm}I{_{x}^{1-\beta}}  \partial _7\underline{L}^r(...))|_{x=A}=0,
\\
\int_{A}^{B} \partial _{11}\overline{L}^r(...)dx-(_x ~\hspace{-0.10cm}I{_{A}^{1-\alpha}}  \partial _5 \overline{L}^r(...)- _A\hspace{-0.10cm}I{_{x}^{1-\beta}}  \partial _7\overline{L}^r(...))|_{x=A}=0;
\end{cases}
\]
when $\tilde{y}(B)$ is free, then
\[
\begin{cases}
\int_{A}^{B} \partial _{12}\underline{L}^r(...)dx-(_B ~\hspace{-0.10cm}I{_{x}^{1-\beta}}  \partial _6 \underline{L}^r(...)- _x\hspace{-0.10cm}I{_{B}^{1-\alpha}}  \partial _4\underline{L}^r(...))|_{x=B}=0,
\\
\int_{A}^{B} \partial _{12}\overline{L}^r(...)dx-(_B ~\hspace{-0.10cm}I{_{x}^{1-\beta}}  \partial _6 \overline{L}^r(...)- _x\hspace{-0.10cm}I{_{B}^{1-\alpha}}  \partial _4\overline{L}^r(...))|_{x=B}=0,
\\
\int_{A}^{B} \partial _{13}\underline{L}^r(...)dx-(_B ~\hspace{-0.10cm}I{_{x}^{1-\beta}}  \partial _7 \underline{L}^r(...)-  _x\hspace{-0.10cm}I{_{B}^{1-\alpha}}  \partial _5\underline{L}^r(...))|_{x=B}=0,
\\
\int_{A}^{B} \partial _{13}\overline{L}^r(...)dx-(_B ~\hspace{-0.10cm}I{_{x}^{1-\beta}}  \partial _7 \overline{L}^r(...)-  _x\hspace{-0.10cm}I{_{B}^{1-\alpha}}  \partial _5\overline{L}^r(...))|_{x=B}=0;
\end{cases}
\]
and when $\tilde{y}(b)$ is free, then
\[
\begin{cases}
\int_{A}^{B} \partial _{14}\underline{L}^r(...)dx-(_A ~\hspace{-0.10cm}I{_{x}^{1-\beta}}  \partial _6 \underline{L}^r(...)- _B\hspace{-0.10cm}I{_{x}^{1-\beta}}  \partial _6\underline{L}^r(...)|_{x=b}=0,
\\
\int_{A}^{B} \partial _{14}\overline{L}^r(...)dx-(_A ~\hspace{-0.10cm}I{_{x}^{1-\beta}}  \partial _6 \overline{L}^r(...)- _B\hspace{-0.10cm}I{_{x}^{1-\beta}}  \partial _6\overline{L}^r(...)|_{x=b}=0,
\\
\int_{A}^{B} \partial _{15}\underline{L}^r(...)dx-(_A ~\hspace{-0.10cm}I{_{x}^{1-\beta}}  \partial _7 \underline{L}^r(...)- _B\hspace{-0.10cm}I{_{x}^{1-\beta}}  \partial _7\underline{L}^r(...))|_{x=b}=0,
\\
\int_{A}^{B} \partial _{15}\overline{L}^r(...)dx-(_A ~\hspace{-0.10cm}I{_{x}^{1-\beta}}  \partial _7 \overline{L}^r(...)- _B\hspace{-0.10cm}I{_{x}^{1-\beta}}  \partial _7\overline{L}^r(...))|_{x=b}=0,
\end{cases}
\]
where
\begin{align*}
(...)=&(x,\underline{y}^{*r}(x),\overline{y}^{*r}(x),_a^{C}\hspace{-0.10cm}D{_{x}^{\alpha}}\underline{y}^{*r},
_a^{C}\hspace{-0.10cm}D{_{x}^{\alpha}}\overline{y}^{*r},~_x^{C}\hspace{-0.10cm}D{_{b}^{\beta}}\underline{y}^{*r},
_x^{C}\hspace{-0.10cm}D{_{b}^{\beta}}\overline{y}^{*r},\\
&\underline{y}^{*r}(a),\overline{y}^{*r}(a),\underline{y}^{*r}(A),\overline{y}^{*r}(A),
\underline{y}^{*r}(B),\overline{y}^{*r}(B),\underline{y}^{*r}(b),\overline{y}^{*r}(b)).
\end{align*}
\end{theorem}


\section{Transversality for Fuzzy Fractional Variational Problems}\label{4}
In this section, for simplicity and without
lose of generality, we consider the following FFVP described by
\begin{equation}\label{f1}
\begin{gathered}
\tilde{J}(\tilde{y})=\int_{a}^{b} \tilde{L}(x,\tilde{y}(x),^{gH-C}\hspace{-0.10cm}_a\mathcal{D}{_{ix}^{\alpha}}\tilde{y}(x)) dx\longrightarrow \rm{extr}, \\
\tilde{y}(a)=\tilde{y}_a
\end{gathered}
\end{equation}
and $\tilde{y}(x)$ intersects the curve $\tilde{\mathit{z}}=\tilde{C}(x)$ for the first time at $b$, i.e. $\tilde{y}(b)=\tilde{C}(b)$. Here, $\tilde{C}(x)$ is a
specified curve. We implicitly assume that all differentiability conditions are met. Note that
in this problem we do not know the end point $x=b$ in advance except that the end point of
$\tilde{y}(x)$ lies on a specified curve.\par
To develop the necessary conditions for this problem, assume that $\tilde{y}^*(x)$ is the desired
function which intersects the curve $\tilde{\mathit{z}}=\tilde{C}(x)$ at $x=b^*$, i.e. $\tilde{y}^*(b^*)=\tilde{C}(b^*)$. Take $\epsilon \in \mathbb{R}$  and define a family of curves
\begin{equation}\label{f2}
\tilde{y}(x)=\tilde{y}^*(x)+\epsilon \tilde{h}(x)
\end{equation}
where $\tilde{h}(x)$ is an arbitrary fuzzy function which satisfy the boundary conditions, i.e., we require
that $\tilde{h}(a)=\tilde{0}$. The  lower bound and upper bound of  $\tilde{y}$ are
\begin{equation}\label{f2,3}
\underline{y}^r(x)=\underline{y}^{r*}(x)+\epsilon \underline{h}^r(x) ~~~{\rm and} ~~~\overline{y}^r(x)=\overline{y}^{r*}(x)+\epsilon \overline{h}^r(x).
\end{equation}
  We further define a set of end--points
\begin{equation}\label{f3}
b=b^*+\epsilon\zeta(b^*)
\end{equation}
 where $\zeta(b^*)$ is a variation in $b^*$.\par
The lower bound and upper bound of $\tilde{J}$ are
\begin{equation}\label{f2,4}
\underline{J}^r(\epsilon)=\int_{a}^{b} \underline{L}^r\left(x,~\underline{y}^r(x),~ \overline{y}^r(x),~_a^{C}\hspace{-0.10cm}D{_{x}^{\alpha}}\underline{y}^r(x),~_a^{C}\hspace{-0.10cm}D{_{x}^{\alpha}}\overline{y}^r(x)\right) dx,
\end{equation}
\begin{equation}\label{f2,5}
\overline{J}^r(\epsilon)=\int_{a}^{b} \overline{L}^r\left(x,~\underline{y}^r(x),~ \overline{y}^r(x),~_a^{C}\hspace{-0.10cm}D{_{x}^{\alpha}}\underline{y}^r(x),~_a^{C}\hspace{-0.10cm}D{_{x}^{\alpha}}\overline{y}^r(x)\right) dx.
\end{equation}
To obtain the extremum of the functional, we substitute equations (\ref{f2,3}) and (\ref{f3}) into
equations (\ref{f2,4}) and (\ref{f2,5}), find the expression for $\frac{d \underline{J}^r}{d\epsilon}$ and $\frac{d \overline{J}^r}{d\epsilon}$, and set $\epsilon,~\frac{d \underline{J}^r}{d\epsilon}$  and  $\frac{d \overline{J}^r}{d\epsilon}$ equal to $0$. This leads to
\begin{multline}\label{f5}
\int_{a}^{b^*}  \left\{\partial_2 \underline{L}^r(...)\underline{h}^r+\partial_3 \underline{L}^r(...)\overline{h}^r+\partial_4 \underline{L}^r(...) ~_a^{C}\hspace{-0.10cm}D{_{x}^{\alpha}}\underline{h}^r+\partial_5 \underline{L}^r(...)~_a^{C}\hspace{-0.10cm}D{_{x}^{\alpha}}\overline{h}^r\right\} dx\\
+\zeta(b^*)  \left[\underline{L}^r\left(b^*,\underline{y}^{r*}(b^*),\overline{y}^{r*}(b^*),_a^{C}\hspace{-0.10cm}D{_{x}^{\alpha}}\underline{y}^{r*}(b^*),_a^{C}\hspace{-0.10cm}D{_{x}^{\alpha}}\overline{y}^{r*}(b^*)\right)\right]=0
\end{multline}
and
\begin{multline}\label{f6}
\int_{a}^{b^*} \left\{\partial_2 \overline{L}^r (...)\underline{h}^r+\partial_3 \overline{L}^r (...)\overline{h}^r+\partial_4 \overline{L}^r (...)~_a^{C}\hspace{-0.10cm}D{_{x}^{\alpha}}\underline{h}^r+\partial_5 \overline{L}^r(...)~_a^{C}\hspace{-0.10cm}D{_{x}^{\alpha}}\overline{h}^r\right\}dx\\
+\zeta(b^*)\left[\overline{L}^r\left(b^*,\underline{y}^{r*}(b^*),\overline{y}^{r*}(b^*),_a^{C}
\hspace{-0.10cm}D{_{x}^{\alpha}}\underline{y}^{r*}(b^*),_a^{C}\hspace{-0.10cm}D{_{x}^{\alpha}}\overline{y}^{r*}(b^*)\right)\right]=0,
\end{multline}
where
$$(...)=(x,\underline{y}^{*r}(x),\overline{y}^{*r}(x),
_a^{C}\hspace{-0.10cm}D{_{x}^{\alpha}}\underline{y}^{*r},_a^{C}\hspace{-0.10cm}D{_{x}^{\alpha}}\overline{y}^{*r}).$$
Note that $b^*$ is still unknown. However, it is a fixed point,   therefore it can be used to
define the right fractional derivative.\par
Using integration by parts and terminal condition $\tilde{y}(a)=\tilde{y}_a$,  equation (\ref{f5}) can be written as
\begin{equation}\label{f7}
\begin{split}
\int_{a}^{b^*} &\left\{\left[\partial_2 \underline{L}^r(...)+_x\hspace{-0.10cm}D{_{b^*}^{\alpha}}\partial_4 \underline{L}^r(...)\right]\underline{h}^r+\left[\partial_3 \underline{L}^r(...)+_x\hspace{-0.10cm}D{_{b^*}^{\alpha}}\partial_5 \underline{L}^r(...)\right]\overline{h}^r\right\} dx\\&
+\left[_x~\hspace{-0.10cm}I{_{b^*}^{1-\alpha}}\partial_4 \underline{L}^r(...)\underline{h}^r(x)+_x~\hspace{-0.10cm}I{_{b^*}^{1-\alpha}}\partial_5 \underline{L}^r(...)\overline{h}^r(x)\right]|_{x=b^*}\\&
+\zeta(b^*).\left(\underline{L}^r(b^*,~\underline{y}^{r*}(b^*),~\overline{y}^{r*}(b^*),~_a^{C}\hspace{-0.10cm}D{_{x}^{\alpha}}\underline{y}^{r*}(b^*),~_a^{C}\hspace{-0.10cm}D{_{x}^{\alpha}}\overline{y}^{r*}(b^*))\right)=0.
\end{split}
\end{equation}
Since the right end point lies on the curve $\tilde{\mathit{z}}=\tilde{C}(x)$,  we have from equations (\ref{f2}) and (\ref{f3}),
\begin{equation}
\tilde{y}^*(b^*+\epsilon\zeta(b^*))+\epsilon\tilde{h}(b^*+\epsilon\zeta(b^*))=\tilde{C}(b^*+\epsilon\zeta(b^*))
\end{equation}
with $r$--level set $[\underline{y}^{r*},~\overline{y}^{r*}]$ where
\begin{equation}\label{f8}
\underline{y}^{r*}(b^*+\epsilon\zeta(b^*))+\epsilon\underline{h}^{r}(b^*+\epsilon\zeta(b^*))=\underline{C}^r(b^*+\epsilon\zeta(b^*)),
\end{equation}
\begin{equation}\label{f9}
\overline{y}^{r*}(b^*+\epsilon\zeta(b^*))+\epsilon\overline{h}^{r}(b^*+\epsilon\zeta(b^*))=\overline{C}^r(b^*+\epsilon\zeta(b^*)).
\end{equation}
Differentiating equations (\ref{f8}) and (\ref{f9}) with respect to $\epsilon$
 and then setting $\epsilon=0$, we obtain
\begin{equation}\label{f10}
\underline{h}^r(b^*)=\zeta(b^*).\left(D\underline{C}^{r}(b^*)-D\underline{y}^{r*}(b^*)\right),
\end{equation}
\begin{equation}\label{f11}
\overline{h}^r(b^*)=\zeta(b^*).\left(D\overline{C}^{r}(b^*)-D\overline{y}^{r*}(b^*)\right).
\end{equation}
Using equations (\ref{f7}),  (\ref{f10}) and (\ref{f11}), the generalized Euler--Lagrange equations and the transversality
conditions can be written, respectively, as
\begin{equation}\label{f14}
\partial_2 \underline{L}^r (...)+_x\hspace{-0.10cm}D{_{b^*}^{\alpha}}\partial_4 \underline{L}^r(...)=0,
\end{equation}
\begin{equation}\label{f15}
\partial_3 \underline{L}^r(...)+_x\hspace{-0.10cm}D{_{b^*}^{\alpha}}\partial_5 \underline{L}^r(...)=0,
\end{equation}
for all $x\in[a,b^*]$, and
\begin{multline*}
\zeta(b^*).\left[\left(
  _x~\hspace{-0.10cm}I{_{b^*}^{1-\alpha}}\partial_4 \underline{L}^r(...)\right).(D\underline{C}^r-D\underline{y}^{r*})\right.\\\left.
+\left(_x~\hspace{-0.10cm}I{_{b^*}^{1-\alpha}}\partial_5 \underline{L}^r(...)\right)(D\overline{C}^r-D\overline{y}^{r*})+\underline{L}^r\right]=0,~~~~\rm{for}~~~~~~~~~~~\it{x=b^*}.
\end{multline*}

In general $\zeta(b^*)$ is not $0$,  therefore the coefficient of $\zeta(b^*)$ in above equation  must be $0$.
This gives the following transversality condition  for the problem
\begin{multline}\label{f13}
\left(_x~\hspace{-0.10cm}I{_{b^*}^{1-\alpha}}\partial_4 \underline{L}^r(...)\right).(D\underline{C}^r-D\underline{y}^{r*})\\
+\left(_x~\hspace{-0.10cm}I{_{b^*}^{1-\alpha}}\partial_5 \underline{L}^r(...)\right).(D\overline{C}^r-D\overline{y}^{r*})+\underline{L}^r=0,~~~~\rm{for}~~~~\it{x=b^*}.
\end{multline}
Following the scheme of obtaining (\ref{f14})-(\ref{f13}) and adapting it to the case under consideration involving (\ref{f6}), one can show that
\begin{equation}\label{f16}
\partial_2\overline{L}^r(...)  +_x\hspace{-0.10cm}D{_{b^*}^{\alpha}}\partial_4 \overline{L}^r(...)=0,
\end{equation}
\begin{equation}\label{f17}
\partial_3 \overline{L}^r(...)+_x\hspace{-0.10cm}D{_{b^*}^{\alpha}}\partial_5 \overline{L}^r(...)=0,
\end{equation}
\begin{multline}\label{f18}
\left(_x~\hspace{-0.10cm}I{_{b^*}^{1-\alpha}}\partial_4 \overline{L}^r(...)\right).(D\underline{C}^r-D\underline{y}^{r*})\\
+\left(_x~\hspace{-0.10cm}I{_{b^*}^{1-\alpha}}\partial_5 \overline{L}^r(...)\right).(D\overline{C}^r-D\overline{y}^{r*})+\overline{L}^r=0,~~~~\rm{for}~~~~~~~~~~~~\it{x=b^*}.
\end{multline}
Eqs.(\ref{f13}) and (\ref{f18}) are called transversality conditions.


\begin{example}\label{exa3}
 Consider the following fuzzy fractional variational problem:
\begin{equation*}
\begin{gathered}
\tilde{J}(\tilde{y}):=\int_{1}^{b} (^{gH-C}\hspace{-0.10cm}_1\mathcal{D}{_{ix}^{\alpha}} \tilde{y})^2 x^3dx \longrightarrow \min,\\ \tilde{y}(1)=<-1,0,1>
\end{gathered}
\end{equation*}
where $b$ is finite and $\tilde{y}(b)$  lies on the curve  $\tilde{y}=\frac{\tilde{2}}{x^{2}}-\tilde{3}$, given that $\tilde{2}=<1,2,3>$ and $\tilde{3}=<2,3,4>$.
\end{example}
$Solution$.
Suppose that $\tilde{y}(x)$ is $[(1) - gH]$-differentiable function. The $r$-level set of $\tilde{J}$
is:
$$[\tilde{J}(\tilde{y})]^r=\left[\int_{1}^{b} (_{1}^C~ \hspace{-0.10cm}D{_{x}^{\alpha}}\underline{y}^r)^2 x^3 dx,~\int_{1}^{b} (_{1}^C~ \hspace{-0.10cm}D{_{x}^{\alpha}}\overline{y}^r)^2 x^3 dx\right].$$ The objective curve is
$\tilde{C}(x)=\frac{\tilde{2}}{x^{2}}-\tilde{3}$  where $\underline{C}^r=\frac{r+1}{x^2}-4+r$ and $\overline{C}^r=\frac{3-r}{x^2}-4-r$, so we have
$D\underline{C}^r=\frac{-2(r+1)}{x^3}$, and $D\overline{C}^r=\frac{-2(3-r)}{x^3}$.\par
From fuzzy Euler--Lagrange equations and transversality conditions in (\ref{f14})-(\ref{f18}) we have
\begin{equation}\label{f19}
\begin{gathered}
_{x}~\hspace{-0.10cm}D{_{b}^{\alpha}} \left(2 (_{1}^C~\hspace{-0.10cm}D{_{x}^{\alpha}}\underline{y}^r) x^3\right)=0,\\
_{x}~\hspace{-0.10cm}D{_{b}^{\alpha}} \left(2( _{1}^C~\hspace{-0.10cm}D{_{x}^{\alpha}}\overline{y}^r) x^3\right)=0,\\
\left(_x~\hspace{-0.10cm}I{_{b}^{1-\alpha}}(2~ _1^{C}\hspace{-0.10cm}D{_{x}^{\alpha} \underline{y}^r})x^3\right).\left(\frac{-2(r+1)}{x^3}-D\underline{y}^r\right)
+(_{1}^C~ \hspace{-0.10cm}D{_{x}^{\alpha}}\underline{y}^r)^2. x^3=0,~~~~\rm{for}~~~~\it{x=b^*},\\
\left(_x~\hspace{-0.10cm}I{_{b}^{1-\alpha}}(2 ~_1^{C}\hspace{-0.10cm}D{_{x}^{\alpha} \overline{y}^r})x^3\right).\left(\frac{-2(r+1)}{x^3}-D\overline{y}^r\right)
+(_{1}^C~ \hspace{-0.10cm}D{_{x}^{\alpha}}\overline{y}^r)^2 x^3=0,~~~~\rm{for}~~~~\it{x=b^*}.
\end{gathered}
\end{equation}
Eqs. (\ref{f19}) must be solved to find the solution to the problem. If we set $\alpha=1,$ then
\begin{equation*}
\begin{gathered}
-\frac{d}{dx} (2 \dot{\underline{y}}^r x^3)=0,
\\
-\frac{d}{dx} (2 \dot{\overline{y}}^r x^3)=0.
\end{gathered}
\end{equation*}
 By virtue of the
classical differential equation theory, we may solve it analytically for fixed $r\in[0,1]$ to arrive at
$$\underline{y}^r=\frac{k_{1}}{x^2}+r-1-k_{1},~\overline{y}^r=\frac{k_{2}}{x^2}+1-r-k_{2},$$  so
$$\dot{\underline{y}}^r=-\frac{2k_{1}}{x^3},~\dot{\overline{y}}^r=-\frac{2k_{2}}{x^3}.$$\par
From transversality conditions (\ref{f13}) and (\ref{f18})
\begin{equation*}
-\frac{4k_{1}}{b^3}b^3\left(-\frac{2(r+1)}{b^3}+\frac{2k_{1}}{b^3}\right)+(-\frac{2k_{1}}{b^3})^2.b^3=0,
\end{equation*}
\begin{equation*}
-\frac{4k_{2}}{b^3}b^3\left(-\frac{2(3-r)}{b^3}+\frac{2k_{2}}{b^3}\right)+(-\frac{2k_{2}}{b^3})^2.b^3=0.
\end{equation*}
Since $b$ is finite we have
$k_{1}=2(r+1)$ and $k_{2}=2(3-r)$.
Hence, we arrive at the extremum  $\tilde{y}(x)$ where $$[\tilde{y}(x)]^r=\left[\frac{2(r+1)}{x^2}-r-3, \frac{2(3-r)}{x^2}+r-5\right]$$ and intersects $\tilde{y}=\frac{\tilde{2}}{x^{2}}-\tilde{3}$ at
$b=\sqrt{2}$.
This solution is shown in Figure \ref{fig3}.

\begin{figure}[tbp]
\centering \includegraphics[scale=.50]{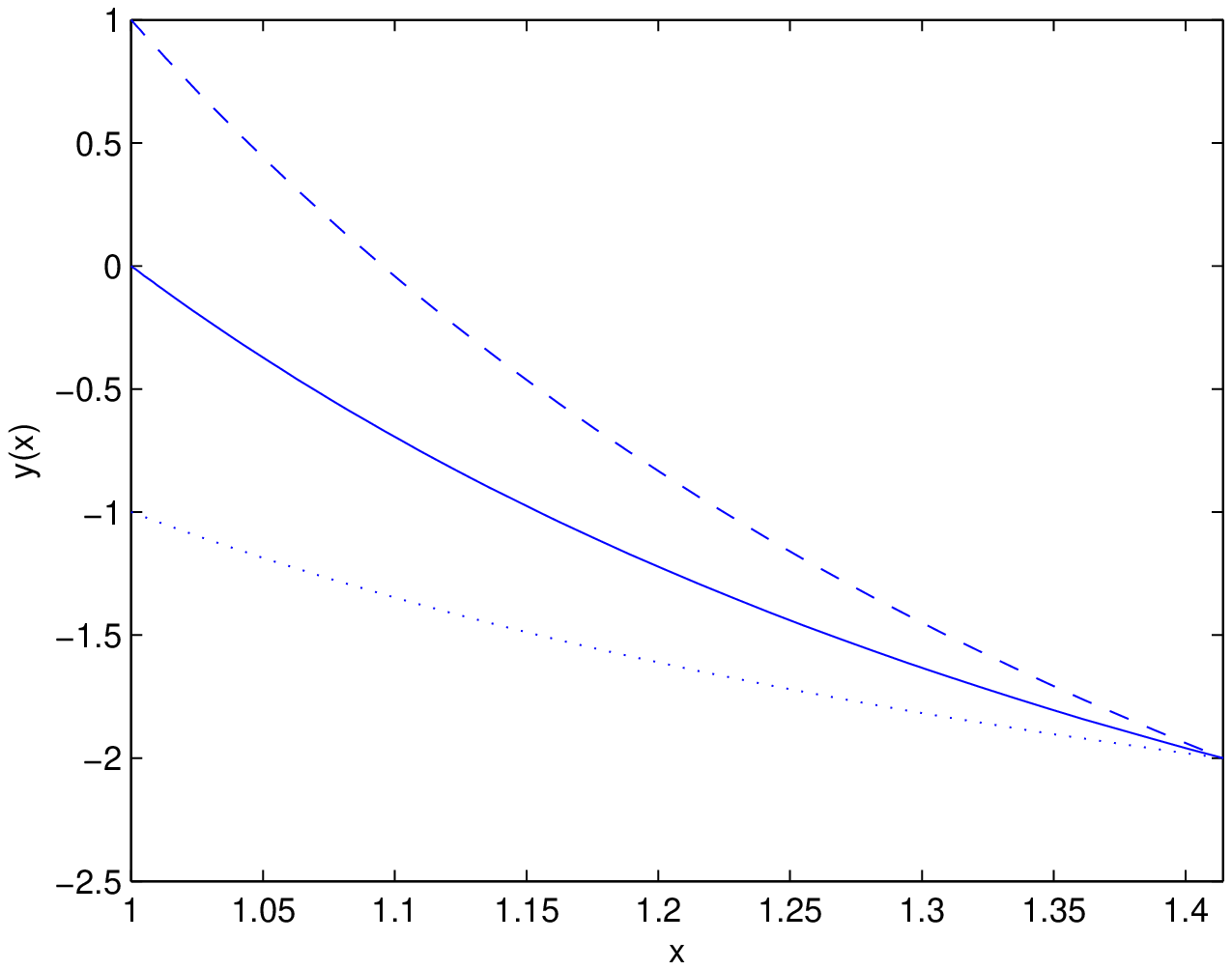}
\caption{solution of Example \ref{exa3}} $\underline{y}^0(x)$ (dotted line), $\overline{y}^0(x)$   (dashed line),~$\underline{y}^1(x)=\overline{y}^1(x)$ (solid line).\label{fig3}
\end{figure}


\section{Conclusion}\label{6}
In this paper we considered a new class of fuzzy fractional functionals of the calculus of
variations containing fuzzy fractional derivatives Liouville--Caputo senses. We  provided necessary optimality conditions for such variational functionals. The Liouville--Caputo  derivative  demands higher conditions
of regularity for differentiability: to compute the fractional derivative of a function in the Liouville--Caputo
sense, we must first calculate its derivative. Liouville--Caputo derivatives are defined only for differentiable
functions while functions that have no first order derivative might have fractional derivatives of
all orders less than one in the Riemann--Liouville sense. In a future paper, we aim to investigate this class of problems in the sense of fuzzy Riemann--Liouville derivative.


\section*{Acknowledgments} 

R. Almeida was supported by Portuguese funds through the CIDMA - Center for Research and Development in Mathematics and Applications, and the Portuguese Foundation for Science and Technology (FCT-Funda\c{c}\~ao para a Ci\^encia e a Tecnologia), within project PEst-OE/MAT/UI4106/2014.




\end{document}